\documentclass[reqno]{amsart}
\usepackage[]{latexsym}
\usepackage[]{amssymb}
\usepackage[]{amsmath}
\usepackage{amsfonts}
\usepackage{amsthm}
\usepackage[all]{xy}
\usepackage[]{mathrsfs}
\usepackage{enumerate}
\usepackage{color}

\newtheorem{theorem}{Theorem}[section]

\newtheorem{thm}[theorem]{Theorem}
\newtheorem{cor}[theorem]{Corollary}
\newtheorem{prop}[theorem]{Proposition}
\newtheorem{lemma}[theorem]{Lemma}

\theoremstyle{definition}

\newtheorem{remark}[theorem]{Remark}

\numberwithin{equation}{theorem}

\DeclareMathAlphabet{\mathpzc}{OT1}{pzc}{m}{it}

\begin{document}

\newcommand{\Trd}{{\mathrm{Trd}}}
\newcommand{\rad}{{\mathrm{rad}}}
\newcommand{\id}{{\mathrm{id}}}
\newcommand{\Ad}{{\mathrm{Ad}}}
\newcommand{\Ker}{{\mathrm{Ker}}}
\newcommand{\wedges}[1]{\d b_1\wedge\ldots\wedge \d b_{#1}}
\newcommand{\rank}{{\mathrm{rank}}}
\renewcommand{\dim}{{\mathrm{dim}}}
\newcommand{\coker}{{\mathrm{Coker}}}
\newcommand{\can}{\overline{\rule{2.5mm}{0mm}\rule{0mm}{4pt}}}
\newcommand{\End}{{\mathrm{End}}}
\newcommand{\Sand}{{\mathrm{Sand}}}
\newcommand{\Hom}{{\mathrm{Hom}}}
\newcommand{\Nrd}{{\mathrm{Nrd}}}
\newcommand{\Srd}{{\mathrm{Srd}}}
\newcommand{\ad}{{\mathrm{ad}}}
\newcommand{\rk}{{\mathrm{rk}}}
\newcommand{\Mon}{{\mathrm{Mon}}}
\newcommand{\disc}{{\mathrm{disc}}}
\newcommand{\Sym}{{\mathrm{Sym}}}
\newcommand{\Skew}{{\mathrm{Skew}}}
\newcommand{\Nrp}{{\mathrm{Nrp}}}
\newcommand{\Trp}{{\mathrm{Trp}}}
\newcommand{\Alt}{{\mathrm{Alt}}}
\newcommand{\Symd}{{\mathrm{Symd}}}
\renewcommand{\dir}{{\mathrm{dir}}}
\renewcommand{\geq}{\geqslant}
\renewcommand{\leq}{\leqslant}
\newcommand{\an}{{\mathrm{an}}}
\newcommand{\alt}{{\mathrm{alt}}}
\renewcommand{\Im}{{\mathrm{Im}}}
\newcommand{\Int}{{\mathrm{Int}}}
\renewcommand{\d}{{\mathrm{d}}}
\newcommand{\qp}[2]{\mbox{$#1 \otimes^\mathrm{qp}#2 $}}
\newcommand{\qf}[1]{\mbox{$\langle #1\rangle $}}
\newcommand{\pff}[1]{\mbox{$\langle\!\langle #1
\rangle\!\rangle $}}
\newcommand{\pfr}[1]{\mbox{$\langle\!\langle #1 ]]$}}
\newcommand{\HH}{{\mathbb H}}
\newcommand{\s}{{\sigma}}
\newcommand{\lra}{{\longrightarrow}}
\newcommand{\ZZ}{{\mathbb Z}}
\newcommand{\NN}{{\mathbb N}}
\newcommand{\FF}{{\mathbb F}}
\newcommand{\PERP}{\mbox{\raisebox{-.5ex}{{\Huge $\perp$}}}}
\newcommand{\Perp}{\mbox{\raisebox{-.2ex}{{\Large $\perp$}}}}
\newcommand{\M}[1]{\mathbb{M}( #1)}
\newcommand{\ind}{{\mathrm{ind}}}
\newcommand{\coind}{{\mathrm{coind}}}

\newcommand{\vf}{\varphi}
\newcommand{\mg}[1]{#1^{\times}}

\newcommand{\ba}{\overline{\rule{2.5mm}{0mm}\rule{0mm}{4pt}}} 
\newcommand{\ra}{\rightarrow}

\title{Symplectic Involutions, quadratic pairs and  function fields of  conics}

\author[A.~Dolphin]{Andrew Dolphin}
\author[A.~Qu\'eguiner-Mathieu]{Anne Qu\'eguiner-Mathieu}
\address{Universiteit Antwerpen, Departement Wiskunde-Informatica, 
Middelheimlaan 1,
2020 Antwerpen, Belgium}
\email{Andrew.Dolphin@uantwerpen.be} 
\address{Universit{\'e}
 Paris 13, Sorbonne Paris Cit{\'e}, LAGA, CNRS (UMR 7539), 99
avenue Jean-Baptiste Cl{\'e}ment, F-93430 Villetaneuse, France}
\email{queguin@math.univ-paris13.fr}
\thanks{The first author is supported by the {Deutsche Forschungsgemeinschaft} project \emph{The Pfister Factor Conjecture in characteristic two} (BE 2614/4) and the FWO Odysseus programme (project \emph{Explicit Methods in Quadratic Form Theory}).
Both authors acknowledge the support of the French Agence Nationale de la Recherche (ANR) under reference ANR-12-BL01-0005. 
}

\begin{abstract} In this paper	 we study symplectic involutions and quadratic pairs that become hyperbolic over the function field of a conic. In particular, we classify them in degree 4 and deduce results on $5$ dimensional  minimal quadratic forms, thus extending to arbitrary fields some results of~\cite{queguiner:conic}, which were only known in characteristic different from $2$. 


\medskip\noindent
\emph{Mathematics Subject Classification (MSC 2010):} 11E39, 11E81, 12F05, 12F10. 
\end{abstract}

\maketitle

\section{Introduction}

Given two projective homogeneous varieties $X$ and $X'$, under the algebraic groups $G$ and $G'$ respectively, one may ask whether $X$ has a rational point over the function field of $X'$. This is a classical question in the theory of algebraic groups. Many particular cases have been studied, leading to beautiful results such as the subform theorem in quadratic form theory~\cite[Prop.~22.4 and Thm.~22.5]{Elman:2008}, the Merkurjev, Wadsworth and Panin index reduction formulae~\cite{MPT1},~\cite{MPT2} and the Karpenko and Karpenko-Zhykhovich theorems on the behavior of an involution under generic splitting of the underlying algebra~\cite{Karp},~\cite{KZ}. 
However, this question is wide open in general. For instance, not much is known on anisotropic quadratic forms that become isotropic over the function field of a given quadric (see for instance~\cite[Chap.~5]{Kahn}). 

An interesting particular case is the following: consider a quaternion algebra $Q$ with norm form $n_Q$ and let $X'$ be the associated conic whose function field will be denoted by $F_Q$. 
By the subform theorem,  an  anisotropic quadratic form that becomes hyperbolic over $F_Q$ is a multiple of the norm form $n_Q$ of $Q$. Conversely, there are quadratic forms that become isotropic over $F_Q$ but do not contain any subform similar to the conic. This observation, due to Wadsworth, led to the notion of $F_Q$-minimal form, introduced by Hoffmann and studied in~\cite{hoffmann:minformsconic} and~\cite{hoffmann:minforms} and also in~\cite{fairve:thesis} in characteristic $2$. 

In~\cite{queguiner:conic}, involutions of the first kind that become hyperbolic over the function field $F_Q$ were studied, under the assumption that the base field has characteristic different from $2$. This provides partial answers to the question above, for some varieties $X$, which are projective homogeneous under groups of type $C$ and $D$. The same question in arbitrary characteristic leads to the study of  symplectic involutions and quadratic pairs that become hyperbolic over $F_Q$. This is the main topic of this paper. In particular, we describe, in characteristic $2$, algebras with a symplectic involution or a quadratic pair that become split and hyperbolic over $F_Q$, and those of degree $4$ that become hyperbolic (non necessarily split) over $F_Q$, see \S~\ref{symp.section} and \S~\ref{sec:qp} respectively. 

One may also ask the same question for orthogonal involutions, however in characteristic $2$ orthogonal involutions  never become hyperbolic. Instead, one must work with the weaker concept of metabolicity, see \cite{dolphin:metainv}. Moreover, the main results in~ \cite{dolphin:direct} reduce
the question to the case of symplectic involutions.   More precisely, an orthogonal involution over a field $F$ of characteristic $2$ becomes metabolic over $F_Q$ if and only if it is an orthogonal sum, in the sense of Dejaiffe~\cite{dejaiffe:orthsums}, of a metabolic orthogonal involution and a symplectic involution that becomes hyperbolic over the same  function field. We therefore only consider symplectic involutions and quadratic pairs in the sequel.

The exceptional isomorphism $B_2\equiv C_2$, and its algebraic consequences described in~\cite[15.C]{Knus:1998}, show that symplectic involutions in degree $4$ are closely related to $5$-dimensional quadratic forms, and hyperbolicity for the involution corresponds to isotropy for the quadratic form. Therefore Faivre's characterisation of $5$-dimensional $F_Q$-minimal forms~\cite[(5.2.12)]{fairve:thesis},  which extends to characteristic $2$ an analogus  statement of Hoffmann, Lewis and Van Geel~\cite[Prop.~4.1]{hoffmann:minforms}, follows easily from our results, see~\S\ref{sec:minform}. 

All results in this paper are already known in characteristic different from $2$, see~\cite{queguiner:conic}. 
Therefore, even though most of our arguments could be written in arbitrary characteristic, we assume throughout the paper that the base field has characteristic $2$, for ease of exposition. The place where we most depart from the characteristic not $2$ case is section~\ref{section:degree4}, notably for the proof of Proposition~\ref{prop:containq}. 

\section{Notations and preliminary observations}\label{section:basics}

Throughout the paper, $F$ is a field of characteristic $2$, $Q$ is the quaternion algebra $Q=[a,b)$ over $F$ and $F_Q$ is the function field of the associated conic (see below for a precise description). 
We refer the reader to~\cite{pierce:1982} as a general reference on central simple algebras, \cite{Knus:1998} for involutions and quadratic pairs and~\cite{Knus:1991} and~\cite{Elman:2008} for hermitian, bilinear and quadratic forms over $F$. For the reader's convenience, we recall below a few basics on forms, involutions and quadratic pairs in characteristic $2$. We also state some lemmas which are used in the proofs of the main results of the paper. 

\subsection{Hermitian, bilinear and quadratic forms} 

Let $(D,\theta)$ be an $F$-division algebra with involution and $h:\,V\times V\rightarrow (D,\theta)$ a hermitian form. If $h(x,y)=0$ for all $y\in V$ implies $x=0$, we say $h$ is nondegenerate. Bilinear forms are hermitian forms over $(F,\id)$. Most of the hermitian and bilinear forms considered below are nondegenerate. 

For $b\in F^\times$, we denote the $2$ dimensional  symmetric bilinear forms 
\begin{multline*}
(x_1,x_2)\times (y_1,y_2)\mapsto x_1y_1+ bx_2y_2
\ \mbox{ and }\ (x_1,x_2)\times (y_1,y_2)\mapsto x_1y_2+ x_2y_1
\end{multline*}
by $\qf{1,b}^{bi}$ and $\HH^{bi}$, respectively. 
For a nonnegative integer m, by an m-fold bilinear Pfister form, we mean a nondegenerate symmetric bilinear form isometric to a tensor product of $m$ binary symmetric bilinear forms representing $1$; we use the notation 
$\pff{b_1,\dots, b_m}\simeq\qf{1,b_1}^{bi}\otimes\dots\otimes\qf{1,b_m}^{bi}$. 

Let $q:\,V\rightarrow F$ be a quadratic form and denote  its polar form by $b_q$, defined by $b_q(x,y)=q(x+y)+q(x)+q(y)$. It is an alternating, hence hyperbolic, bilinear form over $V$. The quadratic form $q$ is called nonsingular if its polar form is nondegenerate. If the polar form $b_q$ has a radical of dimension at most $1$ and the non-zero vectors in this radical are anisotropic then the quadratic form is called nondegenerate. In particular, nonsingular quadratic forms are nondegenerate. Note that both notions are preserved under scalar extensions. 
For all $b_1,b_2,c\in F$, we let $[b_1,b_2]$ be the nonsingular quadratic form $(x,y)\rightarrow b_1x^2+xy+b_2y^2$ and $\qf{c}$ the quadratic form $x\rightarrow cx^2$.
We denote $[0,0]$ by $\HH$.
 If $c$ is non-zero, $\qf{c}$ and $[b_1,b_2]\perp\qf{c}$ are nondegenerate. 

To a quadratic form $q:\,V\rightarrow F$ and a bilinear form $b:\,W\times W\rightarrow F$, one associates the quadratic form, denoted by $b\otimes q$ and defined on $W\otimes V$ by $$(b\otimes q)(w\otimes v)=b(w,w)q(v),\mbox{ see~\cite[p.51]{Elman:2008}}.$$
For any nonnegative integer $m$, by an $m$-fold quadratic Pfister form we mean a quadratic form that is isometric to the tensor product of an $(m-1)$-fold bilinear Pfister form and a nonsingular binary quadratic form representing $1$. We use the notation 
$\pfr{b_1,\dots,b_{m-1},c}\simeq\pff{b_1,\ldots, b_{m-1}}\otimes [1,c]$. 

Given two quadratic spaces $(V,\rho)$ and $(V',\rho')$, we say that $\rho'$ is dominated by $\rho$ if there exists an isometric embedding of $V'$ in $V$, that is $f:\,V'\hookrightarrow V$ such that $\rho(f(x))=\rho'(x)$ for all $x\in V'$. If in addition there exists a quadratic form $\rho''$ such that $\rho=\rho'\perp\rho''$, we say that $\rho'$ is a subform of $\rho$. By~\cite[(7.10)]{Elman:2008}, a nonsingular quadratic  form dominated by $\rho$ is a subform. 
However,  this is not true in general. For example, the form $[1,a]\perp\qf{b}$ is dominated by $\pfr{a,b}=[1,a]\perp\qf{b}[1,a]$ but it is not a subform.

Even though  one cannot cancel  quadratic forms in general in  characteristic $2$, one can always cancel  nonsingular forms, and in particular hyperbolic planes, see~\cite[(8.4)]{Elman:2008}. 
Hence the Witt group of nonsingular quadratic forms over $F$ is well defined and denoted by $W_q(F)$. 
Moreover, the action taking a  tensor product of a symmetric bilinear form with a quadratic form gives $W_q(F)$  the  structure of a $W(F)$-module, where $W(F)$ is the Witt ring of symmetric bilinear forms over $F$.
We let $I^m_q(F)$ be the ideal generated by $m$-fold quadratic Pfister forms over $F$. Note that the exponent in this notation differs by $1$ from the exponent in~\cite{Knus:1998}.

We recall some well-known and useful identities:\begin{lemma}
For $b_1,b_2,c_1,c_2\in F$ and $x\in F^\times$
we have  
\begin{eqnarray}\label{eqnarray:nonsingiso}
[b_1,b_2]\perp[c_1,c_2]\simeq [b_1+c_1, b_2] \perp[c_1,b_2+c_2]\,,
\end{eqnarray}
\begin{eqnarray}
\label{qfidentity}
[1,b_1]\perp[1,b_2]\simeq [1,b_1+b_2]\perp\HH\,,
\end{eqnarray}
\begin{eqnarray}\label{eqnarray:multiply}
x[b_1,b_2]\simeq [x\cdot b_1, x^{-1}\cdot b_2]\,,\end{eqnarray}
\begin{eqnarray}\label{eqnarray:singiso}
[b_1,b_2]\perp\qf{c_1}\simeq[b_1+c_1,b_2]\perp\qf{c_1}\end{eqnarray}
and if $[b_1,b_2]\perp\qf{c_1}$ is isotropic, then 
 \begin{eqnarray}\label{eqnarray:isoiso}
 [b_1,b_2]\perp\qf{c_1}\simeq \HH\perp\qf{c_1}\,.\end{eqnarray}
 \end{lemma}
 \begin{proof} The first four isometries are easy to check. For the final isometry, see
 \cite[\S 2]{HoffmannLaghribi:qfpfisterneigbourc2}.
 \end{proof}
  Note that (\ref{eqnarray:isoiso}) provides examples where  cancellation does not hold: in general, one cannot cancel  $\qf{c_1}$. 
 We will also use a particular case of a general result of Hoffmann and Laghribi~\cite[(3.9)]{HoffmannLaghribi:qfpfisterneigbourc2}: 
\begin{prop}\label{prop:nsc}
Consider $b_1,b_2,c_1,c_2\in F$ and $d\in F^\times$ such that 
$[b_1,b_2]\perp\qf{d}\simeq [c_1,c_2]\perp\qf{d}$. 
For all $d'\in F$, there exists $d''\in F$ such that 
$$[b_1,b_2]\perp d[1,d']\simeq [c_1,c_2]\perp d[1,d''].$$
\end{prop}

For all $\lambda,\mu\in F^\times$, the sum of  bilinear Pfister forms $\pff{\lambda}\perp\pff{\mu}\perp\pff{\lambda\mu}$ is isometric to $\pff{\lambda,\mu}\perp\HH^{bi}$, hence we have the following lemma.
 \begin{lemma} 
 \label{Pfisteridentity}
 Let $\rho$ be a quadratic form whose Witt class belongs to $I_q^m(F)$, and $\lambda,\mu\in F^\times$. 
 We have $$\pff{\lambda}\otimes \rho\perp\pff{\mu}\otimes\rho\equiv \pff{\lambda\mu}\otimes\rho\mod I_q^{m+2}(F).$$
 \end{lemma} 
 The following is a well known and useful property of quadratic Pfister forms: 
 \begin{lemma}\label{lemma:round}
Let $\pi$ be an $m$-fold quadratic Pfister form over $F$ and let $c\in F^\times$ be an element represented by $\pi$. Then $\pi\simeq c\pi$ and for all $d\in F^\times$ we have $\pff{d}\otimes\pi\simeq \pff{cd}\otimes \pi$.
\end{lemma}
\begin{proof}
See  \cite[(9.9)]{Elman:2008} for the isometry $\pi\simeq c\pi$. The second assertion follows immediately, since $\pff{d}\otimes \pi\simeq \pi\perp d\pi$. \end{proof}

From this	 we deduce a characteristic $2$ analogue of the well-known `common slot lemma' (see~\cite[(6.16)]{Elman:2008} for example).  For $1$-fold quadratic Pfister forms it is proved in~\cite[Lemma 6]{arason:rel}. 
\begin{lemma}\label{prop:commonslot}
Let  $\pi$ and $\pi'$ be $m$-fold quadratic Pfister forms over $F$ such that for some $c,c'\in F^\times$ we have that $\pff{c}\otimes \pi\simeq \pff{c'}\otimes \pi'$. Then there exists an element $d\in F^\times$ such that $\pff{c}\otimes \pi\simeq \pff{d}\otimes \pi\simeq\pff{d}\otimes \pi'\simeq\pff{c'}\otimes \pi'$.
\end{lemma}
\begin{proof}
If one of $\pff{c}\otimes \pi$ or $\pff{c'}\otimes \pi'$ is hyperbolic then they both are and we may take $d=1$.  Therefore we may assume both are anisotropic, and in particular we may assume $\pi$ and $\pi'$ are anisotropic. Consider the hyperbolic quadratic form 
 $$\pff{c}\otimes \pi\perp \pff{c'}\otimes \pi' = \pi\perp \pi' \perp c\pi \perp c'\pi' \simeq (2^{m+2})\times \HH \,.$$ 
 Taking the orthogonal sum of this form and $ c\pi \perp c'\pi '$ and using Witt cancellation gives 
 $$ \pi\perp \pi'\simeq c\pi \perp c'\pi'\,.$$ 
 As $\pi$ and $\pi'$ both represent $1$ the form $ \pi\perp \pi'$ is isotropic, and hence  $ c\pi \perp c'\pi'$  is isotropic. As $\pi$ and $\pi'$ are anisotropic there exists an element $d\in F^\times$ represented by both $ c\pi$ and $c'\pi'$. Therefore there exists an element $s\in F^\times$ represented by $\pi$ and an element $t\in F^\times$ represented by $\pi'$ such that $d=cs=c't$.
 The result then follows from (\ref{lemma:round}).
\end{proof}

\subsection{Algebras with involution} 
\label{AI.sec}
Throughout $A$ denotes a central simple algebra over $F$. 
The index of $A$ is the degree of its division part and the coindex of $A$ is defined by $\coind(A)=\deg(A)/\ind(A)$. 
That is, $A\simeq M_r(D)$, where $r$ is the coindex of $A$ and $D$ is a division algebra Brauer equivalent to $A$. 
All the involutions considered in this paper are $F$-linear. 

If the algebra $A$ is split, that is, $A\simeq \End_F(V)$, an $F$-linear involution on $A$ is the adjoint of a nondegenerate symmetric bilinear form $b:V\times V\rightarrow F$, uniquely defined up to a scalar factor. We denote this algebra with involution by $\Ad_b$. The involution is symplectic if $b$ is alternating, and orthogonal if $b$ is symmetric and non-alternating. 

We use the notations $\Sym(A,\sigma)=\{x\in A,\ \sigma(x)=x\}$ and $\Symd(A,\sigma)=\{\sigma(x)+x,\ x\in A\}$ for the sets of symmetric and symmetrised elements, respectively. Contrary to the case of  fields of  characteristic different from $2$, $\Symd(A,\sigma)$ is a strict subset of $\Sym(A,\sigma)$. More precisely, both are subvector spaces of $A$ of  dimension $\frac{n(n-1)}{2}$ and $\frac{n(n+1)}{2}$ respectively, where $n$ is the degree of $A$. 
One may prove that the involution $\sigma$ is symplectic if and only if all symmetric elements have reduced trace $0$, or equivalently $1$ is a symmetrised element~\cite[(2.5),(2.6)]{Knus:1998}. In particular, in characteristic $2$, a tensor product of involutions with at least one symplectic factor always is symplectic. 

The algebra with involution $(A,\sigma)$ is called isotropic if there exists a non-zero element $x\in A$ such that $\sigma(x)x=0$. 
If $A$ contains an idempotent $e$ such that $\sigma(e)=1-e$, the involution $\sigma$ is called hyperbolic. Assume $A\simeq M_r(D)$ and $\sigma$ is the adjoint of a nondegenerate hermitian form $h$ with respect to $(D,\theta)$ for some $F$-linear involution $\theta$ in $D$. The involution $\sigma$ is isotropic (respectively hyperbolic) if and only if the hermitian form $h$ is isotropic (respectively hyperbolic). For hyperbolicity this is explained in~\cite[(6.7)]{Knus:1998}; for isotropy, one may easily extend to nondegenerate hermitian forms the argument given in~\cite[(3.2)]{dolphin:quadpairs} for nondegenerate symmetric bilinear forms. In particular, if $(A,\sigma)$ is hyperbolic, then the algebra $A$ has even coindex. 

There is a unique nondegenerate alternating bilinear form of a given rank over $F$, up to isomorphism, and this form is hyperbolic~\cite[Prop. 1.8]{Elman:2008}. Therefore, up to isomorphism, there is a unique symplectic involution on a split algebra and it is hyperbolic. Conversely, since a hyperbolic symmetric bilinear form is alternating, non-alternating forms are not hyperbolic, and neither are orthogonal involutions. 
In the non-split case, since nondegenerate hyperbolic forms of the same rank are isomorphic, a central simple algebra of even coindex admits a unique hyperbolic involution up to isomorphism, and this involution is of symplectic type.

An $F$-quaternion algebra is a central simple $F$-algebra of degree $2$. 
Any $F$-quaternion algebra has a basis $(1,u,v,w)$ such that
$$u(1+u) =r, v^2=s\textrm{ and }w=uv=v(1+u)\,$$
for some  $r\in F$ and $s\in F^\times$
 (see  \cite[Chap.~IX, Thm.~26]{Albert:1968}); any such basis is called a quaternion basis.
 Conversely, for  $r\in F$ and $s\in F^\times$
 the above relations uniquely determine an $F$-quaternion algebra, which we denote by $[r,s)$. 

Let $H=[r,s)$ be an $F$-quaternion algebra, with quaternion basis $(1,u,v,w)$. 
By~\cite[(2.21)]{Knus:1998}, the map $H\rightarrow H,$  $x\mapsto \bar x=\Trd_H(x)+x$ is the unique symplectic involution on $H$. 
It is called the canonical involution of $H$, and determined by the conditions that $\overline{u}=1+u$ and $\overline{v}=v$.  
The symmetric elements in $(H,\can)$ are called pure quaternions and we use the notation 
$$H^0=\Sym(H,\can)=F\oplus Fv\oplus Fw.$$ Since the involution is symplectic, we have $\Symd(H,\can)=F\subset H^0$. 
 
An easy computation gives  the following lemma, which will be used in \S~\ref{section:degree4}.  \begin{lemma}\label{prop:quatbasischange}
 Let $H=[r,s)$ be an $F$-quaternion algebra with quaternion basis $(1,u,v,w)$. 
 \begin{enumerate}[$(1)$]
 \item For all $\lambda,\mu\in F$, the quaternion algebra $H$ admits a quaternion basis $(1,u',v',w')$ with $u'=u+\lambda v+\mu w$, and $v'=v$. 
 \item For all $\lambda, \mu \in F$ such that $(\lambda v+\mu w)^2\in F^\times$, the quaternion algebra $H$ admits a quaternion basis $(1,u',v',w')$ with $u'=u$ and $v'=\lambda v+\mu w$. 
 \end{enumerate}
 \end{lemma}
 
 The reduced norm of $H$ defines a nonsingular $4$-dimensional quadratic form on the $F$-vector space $H$, which is denoted by $n_H$. 
Computing this form with the  quaternion basis associated to the representation $H=[r,s)$, one gets  $n_H=\pfr{r,s}=[1,s]\perp r[1,s]$.  
The restriction of the norm form to pure quaternions leads to a nondegenerate conic, for which we use the notation 
$n_H^0=\qf{1}\perp s[1,r]$. This form is similar to $[1,r]\perp\qf{s}$. Its function field is denoted by $F_H$. 

Throughout the paper, $Q$ is a fixed quaternion algebra, and we let $Q=[a,b)$ for some $a\in F$ and $b\in F^\times$. We assume in addition that $Q$ is division, so that its norm form $n_Q=\pfr{a,b}$ is anisotropic. The field $F_Q$ is the function field of the conic $n_Q^0$, which is similar to $[1,a]\perp\qf{b}$. By Amitsur's theorem, an $F$-quaternion algebra is split over $F_Q$ if and only if it is either split or isomorphic to $Q$, see~\cite[Remark 5.4.9]{Gille:2006}. 
Moreover, we have: 
\begin{lemma} 
\label{prop:pfisterneigh}
An anisotropic quadratic form $\varphi$ of even dimension is hyperbolic over $F_Q$ if and only if there exists a symmetric bilinear form $b$ such that $\varphi\simeq b\otimes n_Q$. 
\end{lemma} 
\begin{proof}
Applying~\cite[(22.17)]{Elman:2008} several times if necessary, one may check that $\varphi$ is hyperbolic over $F_Q$ if and only if it is hyperbolic over the function field $F(n_Q)$ of the norm form of $Q$. Hence, the lemma is an immediate consequence of the multiplicative subform theorem~\cite[(23.6)]{Elman:2008}. 
\end{proof}

Let $(A,\sigma)$ be a central simple algebra, endowed with an involution of symplectic type. 
We say that $(A,\sigma)$ contains $(Q,\can)$ if $A$ contains a $\sigma$-stable subalgebra isomorphic to $Q$ on which $\sigma$ acts as the canonical involution of $Q$. If this is the case, considering the centraliser $B$ of this subalgebra in $A$, we get a decomposition $$(A,\sigma)\simeq(Q,\can)\otimes (B,\tau).$$ 
The quaternion algebra $Q$ is split by $F_Q$, and the canonical involution is symplectic; therefore $(Q,\can)_{F_Q}$ is hyperbolic and it follows that any $(A,\sigma)$ containing $(Q,\can)$ is also hyperbolic. 
The converse does not hold in general, as we shall explain in \S\,\ref{symp.section}. 

\subsection{Quadratic Pairs} The basic results on quadratic pairs that we recall here can be found in  \cite[\S 5]{Knus:1998}. 
In arbitrary characteristic, algebraic groups of type $D$ can be described in terms of quadratic pairs. 
A quadratic pair on a central simple algebra $A$ is a couple $(\sigma,f)$, where $\sigma$ is an $F$-linear involution on $A$, with $\Sym(A,\sigma)$ of dimension $\frac{n(n+1)}{2}$, and $f$ is a so-called semi-trace on $(A,\sigma)$, that is an $F$-linear map $f:\,\Sym(A,\sigma)\ra F$ such that $f(x+\sigma(x))=\Trd_A(x)$ for all $x\in A$. In characteristic different from $2$, the dimension condition guarantees that the involution is of orthogonal type, and one may check that there is a unique semi-trace on $(A,\sigma)$ given  by $f(x)=\frac 12 \Trd_A(x)$ for all $x\in A$. Therefore quadratic pairs and orthogonal involutions are equivalent notions when the  characteristic is not $2$. 
Conversely, in characteristic $2$, the existence of a semi-trace implies $\sigma$ is symplectic. Indeed, it implies that the reduced trace vanishes on $\Sym(A,\sigma)$, since $\Trd_A(c)=f(c+\sigma(c))=f(2c)=0$ for all $c\in \Sym(A,\sigma)$. 

Let $(V,\rho)$ be a nonsingular quadratic space over the field $F$. The polar form $b_\rho$ of $\rho$ induces a symplectic (and hyperbolic) involution $\ad_{b_\rho}$ on $A=\End_F(V)$, and one may prove that $(\End_F(V),\ad_{b_\rho})\simeq (V\otimes V,\varepsilon)$, where $\varepsilon$ is the exchange involution, defined by $\varepsilon(x\otimes y)=y\otimes x$. Moreover, there exists a unique semi-trace $f$ defined on $\Sym(V\otimes V, \varepsilon)$ and satisfying $f(x\otimes x)=\rho(x)$ for all $x\in V$. 
Under the isomorphism above, $f$ defines a semi-trace $f_\rho$ on $\Sym(\End_F(V),\ad_{b_\rho})$. The quadratic pair $(\ad_{b_\rho},f_\rho)$ is called the adjoint of $\rho$, and we use the notation $\Ad_{\rho}$ for the algebra with quadratic pair $(\End_F(V),\ad_{b_\rho},f_\rho)$. As explained in~\cite[(5.11)]{Knus:1998}, any quadratic pair on a split algebra $\End_F(V)$ is the adjoint of a nonsingular quadratic form $\rho$ on $V$. 

The notions of isotropy and hyperbolicity of quadratic forms extend to quadratic pairs; see~\cite[(6.5),(6.12)]{Knus:1998} for the definitions. 

Let $(B,\tau)$ be an algebra with involution, and $(A,\sigma,f)$ an algebra with quadratic pair. 
Since $\sigma$ is symplectic, the involution $\tau\otimes \sigma$ also is symplectic. 
Moreover, there exists a unique semi-trace  $f_\star$ on $\Sym(B\otimes A,\tau\otimes \sigma)$ satisfying $$f_\star(b\otimes a)=\Trd_B(b)\otimes f(a)\mbox{ for all }b\in\Sym(B,\tau)\mbox{ and }a\in\Sym(A,\sigma)\mbox{~\cite[(5.18)]{Knus:1998}.}$$
This defines a tensor product of $(B,\tau)$ and $(A,\sigma, f)$ giving $(B\otimes A,\tau\otimes \sigma,f_\star)$, which we denote by $(B,\tau)\otimes(A,\sigma,f)$, and one may check it corresponds to the usual tensor product in the split case, that is $\Ad_b\otimes \Ad_\rho=\Ad_{b\otimes\rho}$ for all nondegenerate symmetric bilinear forms $b$ and nonsingular quadratic forms $\rho$, see \cite[(5.19)]{Knus:1998}. 
 By \cite[(5.3)]{dolphin:totdecomp}, the tensor product of algebras with involution and the tensor product of an algebra with involution and an algebra with quadratic pair are mutually associative. In particular, for an $F$-algebra with quadratic pair $(A,\s,f)$ and $F$-algebras with involution of the first kind $(B,\tau)$ and $(C,\gamma)$, we may write $(C,\gamma)\otimes(B,\tau)\otimes(A,\s,f)$ without any ambiguity.

If $b$ is an alternating, hence  hyperbolic, bilinear form, then $b\otimes \rho$ is a hyperbolic quadratic form and up to isomorphism only  depends on $\dim(\rho)$. Similarly, one may check that $f_\star$ does not depend on $f$ if $\tau$ is symplectic. Indeed, if $\tau$ is symplectic, then $\Trd_B(b)=0$ for all $b\in\Sym(B,\tau)$ and $f_\star$ is the unique semi-trace on $(B\otimes A,\tau\otimes \sigma)$ such that $f_\star(b\otimes a)=0$ for all $b\in \Sym(B,\tau)$ and $a\in\Sym(A,\sigma)$. We call this semi-trace the  canonical semi-trace  on $(B,\tau)\otimes(A,\sigma)$ and denote 
the resulting $F$-algebra with quadratic pair by $(B,\tau)\boxtimes(A,\sigma)$. We get the  following result  (see also \cite[(5.4)]{dolphin:totdecomp}).

\begin{lemma}
\label{decqp.lem} 
Given two algebras with symplectic involution $(A,\sigma)$ and $(B,\tau)$, we have 
$$(A,\sigma)\boxtimes (B,\tau)\simeq (A,\sigma)\otimes (B,\tau,g)\simeq (B,\tau)\otimes (A,\sigma,f) \simeq (B,\tau)\boxtimes(A,\sigma)\,,$$
for all semi-traces $f$ on $(A,\sigma)$ and $g$ on $(B,\tau)$. \end{lemma} 
Moreover, for any algebra with involution of the first kind $(C,\gamma)$ we have 
$$\bigl((C,\gamma)\otimes (B,\tau)\bigr)\boxtimes (A,\sigma)\simeq (C,\gamma)\otimes(B,\tau)\otimes(A,\sigma,f)\simeq (C,\gamma)\otimes\bigl((B,\tau)\boxtimes(A,\sigma)\bigr)\,.$$ Therefore, we will use the notation $(C,\gamma)\otimes(B,\tau)\boxtimes (A,\sigma)$ for this tensor product. 
If $(A,\sigma)$ and $(B,\tau)$ are isomorphic to $(Q,\can)$, then the $F$-algebra with  quadratic pair $(A,\sigma)\boxtimes(B,\tau)$ is the adjoint of the norm form $n_Q$, as we now prove.
\begin{lemma}
\label{decAdnQ.lem}
$(Q,\can)\boxtimes(Q,\can)\simeq \Ad_{n_Q},$ 
where $n_Q$ is the norm form of $Q$.  
\end{lemma} 
\begin{proof}
It only remains to check the second isomorphism, which follows from~\cite[Exercise 22, Chap I]{Knus:1998}. 
Indeed, consider a quaternion basis $(1,u,v,w)$ of $Q$; we have $\bar u=1+u$. Therefore 
$(Q,\can)\boxtimes(Q,\can)$ is isomorphic to $\End_F(Q)$, endowed with the quadratic pair adjoint to the quadratic form $\rho$ defined by $\rho(x)=\Trd_Q(\bar x u x)$. Since $\Trd_Q(\bar x u x)=\Trd_Q(ux\bar x)=\Trd_Q(un_Q(x))=n_Q(x)$, we get $\rho={n_Q}$ as required. 
\end{proof} 

Let $(A,\sigma,f)$ be an algebra with quadratic pair. We say that $(A,\sigma,f)$ contains $(Q,\can)$ if there exists an algebra with symplectic involution $(B,\tau)$ such that $(A,\sigma,f)\simeq(Q,\can)\boxtimes(B,\tau)$. 
By Lemma \ref{decqp.lem}, this is equivalent to 
 $(A,\sigma,f)\simeq (Q,\can)\otimes(B,\tau,g)$ for any choice of semi-trace $g$ on $(B,\tau)$. 
 This condition is stronger than the existence of a $\sigma$-stable subalgebra in $A$, isomorphic to $(Q,\can)$. Indeed, consider any non-singular $4$-dimensional quadratic form $\rho$ with non-trivial discriminant. By~\cite[(7.10)]{Knus:1998}, $\Ad_\rho$ is indecomposable. Nevertheless, the underlying algebra with involution is $(M_4(F),\ad_{2\HH^{bi}})\simeq (Q,\can)\otimes (Q,\can)$ as an algebra with involution. With our definition, this $\Ad_\rho$ does not contain $(Q,\can)$. 
Since $(Q,\can)$ is hyperbolic over $F_Q$, if $(A,\sigma,f)$ contains  $(Q,\can)$  then it is also hyperbolic over $F_Q$ by \cite[A.5]{Tignol:galcohomgps}. As we shall see in~\S\,\ref{sec:qp}, the converse does hold in degree $4$ for anisotropic quadratic pairs.

\section{Symplectic involutions and minimal forms}
\label{symp.section}

In this section we give a complete description of the $(A,\sigma)$ that are hyperbolic over $F_Q$ if either $A$ is split by $Q$ or $A$ has degree $4$; as we will explain, they do not necessarily contain $(Q,\can)$. In \S~\ref{sec:minform}, we describe the relation with $F_Q$-minimal quadratic forms. 

\subsection{Split and hyperbolic over $F_Q$}

\label{sec:splithyp}

Let us first assume that $(A,\sigma)$ is both split and hyperbolic over $F_Q$. 
As recalled in \S\,\ref{AI.sec}, by Amitsur's theorem~\cite[\S 5.4]{Gille:2006}, this implies $A$ is either split or Brauer equivalent to $Q$. Up to isomorphism, a split algebra admits a unique symplectic involution, namely the hyperbolic involution. Hence if $A_{F_Q}$ is split then $(A,\sigma)_{F_Q}$ is split and hyperbolic, with no additional condition on the  involution. That is, an algebra with symplectic involution $(A,\sigma)$ is split and hyperbolic over $F_Q$ if and only if the algebra is either split or Brauer equivalent to $Q$. We now prove: 

\begin{prop}\label{prop:splithyp}
 Let $(A,\sigma)$ be an $F$-algebra with symplectic involution, such that $A\otimes_F F_Q$ is split. 
  \begin{enumerate}[$(1)$]
 \item If $A$ is split then $(A,\sigma)$ is hyperbolic. Moreover, $(A,\sigma)$ contains $(Q,\can)$ if and only if $\deg A \equiv 0 \mod 4$. 
  \item If $A$ is Brauer-equivalent to $Q$, then there exists a symmetric bilinear form $b$ over $F$ such that $(A,\sigma)\simeq \Ad_b\otimes (Q,\can)$. In particular, $(A,\sigma)$ contains $(Q,\can)$ and $(A,\sigma)_{F_Q}$ is hyperbolic. 
 \end{enumerate}
\end{prop}
\begin{proof}
$(1)$ If $A$ is split and $(A,\sigma)=(B,\tau)\otimes(Q,\ba)$, then $B$ is Brauer-equivalent to $Q$. 
Therefore $B$ has even degree and $\deg(A)=2\deg(B)\equiv 0\mod 4$. 
Assume conversely that $\deg(A)=4r$ for some integer $r\geq 1$. 
Since $\can\otimes\ba$ is symplectic, the tensor product $(Q,\can)\otimes (Q,\ba)$ is split,  hyperbolic and has degree $4$. Therefore for any nondegenerate $r$-dimensional symmetric bilinear form $b$, 
$$(A,\sigma)\simeq \Ad_b\otimes (Q,\can)\otimes(Q,\ba),$$
since these two algebras with involution are both split, symplectic, hyperbolic and of the same degree. This proves $(A,\sigma)$ contains $(Q,\ba)$. 

$(2)$
Let $V$ be a finite dimensional right $Q$-vector space such that $A\simeq\End_{Q}(V)$. Consider a hermitian form $h:\,V\times V\ra Q$ over $(Q,\ba)$ such that $\sigma=\ad_h$ (see~\cite[(4.2)]{Knus:1998}). Since $\sigma$ is symplectic, $h$ is alternating, that is $$\mbox{for all }v\in V,\ h(v,v)\in \Symd(Q,\ba)=F.$$ Moreover, by  \cite[Chap.~I, (6.2.4)]{Knus:1991} there exists an orthogonal basis  $(v_1,\ldots,v_r)$  of $(V,h)$. By restriction to the $F$-vector space $U=Fv_1\oplus\ldots\oplus Fv_r$, $h$ induces a symmetric bilinear form $b$ over $F$, and the natural isomorphism of $F$-spaces $U\otimes_FQ\simeq V$ induces an isomorphism of $F$-algebras with involution $\Ad_b\otimes (Q,\ba)\simeq (A,\sigma).$
\end{proof}

\begin{remark}
 \begin{enumerate}[$(1)$]
\item The previous proposition provides a description of all algebras with involution of degree $\deg(A)\equiv 2\mod 4$ which are hyperbolic over $F_Q$. Indeed, such an algebra has index at most $2$, hence it is either split or Brauer-equivalent to a quaternion division algebra. If it is not split it has odd co-index, hence no hyperbolic involution. Therefore if it admits a symplectic involution $\sigma$ which becomes hyperbolic over $F_Q$, it has to be split and hyperbolic over $F_Q$. 
\item The proposition also shows that any algebra with {\em anisotropic} symplectic involution which is split and hyperbolic over $F_Q$ actually is Brauer equivalent to $Q$, and does contain $(Q,\ba)$. 
\end{enumerate}
\end{remark}

\subsection{Degree $4$}\label{section:degree4}

Throughout this section $A$ is a central simple $F$-algebra of degree $4$, endowed with a symplectic involution. Since $Q\otimes_F Q\simeq M_4(F)$ is split, a direct computation in the Brauer group of $F$ shows that the algebra $A$ contains $Q$, that is decomposes as $Q'\otimes_F Q$ for some quaternion algebra $Q'$ over $F$, if and only if the tensor product $A\otimes_F Q$ has index at most $2$, that is, $A\otimes_F Q\simeq M_4(Q')$. When these conditions are satisfied the following proposition characterises the anisotropic symplectic involutions that become hyperbolic over $F_Q$. 

\begin{prop}\label{prop:containq}
Let $(A,\sigma)$ be a degree $4$ algebra with anisotropic symplectic involution. We assume $A\simeq Q\otimes_F Q'$ for some quaternion algebra $Q'$ over $F$. The involution $\sigma_{F_Q}$ is hyperbolic if and only if there exists an $F$-linear involution $\tau$ on $Q'$ such that $$(A,\sigma)\simeq (Q',\tau)\otimes(Q,\ba).$$
\end{prop} 

From this, we deduce a complete description of anisotropic symplectic involutions in degree $4$ that become hyperbolic over $F_Q$: 
\begin{thm}\label{thm:sympclass}
Let $(A,\sigma)$ be a degree $4$ algebra with anisotropic symplectic involution. 
The involution $\sigma_{F_Q}$ is hyperbolic if and only if either
\begin{enumerate}[$(a)$]
\item $(A,\s)\simeq (Q',\tau)\otimes (Q,\can)$ for some quaternion $F$-algebra with involution $(Q',\tau)$, or 
\item  $(A,\s)\simeq \Ad_{\pff{\lambda}}\otimes (Q',\ba)$ for some quaternion $F$-algebra $Q'$ and some $\lambda\in F^\times$ such that $Q\otimes_F Q'$ is a division algebra and $\pff{\lambda}\otimes n_Q\simeq \pff{\lambda}\otimes n_{Q'}$.
\end{enumerate}
\end{thm}
Case (b) provides examples of algebras with anisotropic symplectic involutions that become hyperbolic over $F_Q$ and do not contain $(Q,\ba)$. Indeed, since the algebra  $Q\otimes_FQ'$ is division, $M_2(Q')$ does not contain $Q$. It also follows from this observation that cases (a) and (b) are mutually exclusive. 

The proofs of Proposition~\ref{prop:containq} and Theorem~\ref{thm:sympclass} both rely on the fact that symplectic involutions in degree $4$ are classified by a  relative invariant given by a quadratic form, as we now  recall. For any symplectic involution $\gamma$ on a biquaternion algebra $H\otimes_F H'$, we denote the reduced Pfaffian norm and trace of $(H\otimes_F H',\gamma)$ by $\Nrp_\gamma$ and $\Trp_\gamma$ respectively, see~\cite[p.19]{Knus:1998} for a definition. 
In particular, $\Nrp_{\gamma}$ is a nonsingular quadratic form on the $6$-dimensional vector space $\Symd(H\otimes_F H',\gamma)$,  whose class in $I^2_q(F)$ is equal to the class of $n_H\perp n_{H'}$ (see \cite[(16.8) and (16.15)]{Knus:1998}). Hence $\Nrp_\gamma$ is an Albert form of the biquaternion algebra $H\otimes_F H'$. 

Consider another symplectic involution $\gamma'$ on $H\otimes_F H'$.  By~\cite[(2.7)]{Knus:1998}, there exists an invertible  element $x\in \Symd(H\otimes_F H',\gamma)$ such that $\gamma'=\Int(x)\circ\gamma$. Using~\cite[(16.15)]{Knus:1998}, one may compute the difference (or sum) of the two Pfaffian reduced norms: $\Nrp_{\gamma'}\perp \Nrp_\gamma=\pff{\Nrp_\gamma(x)}\otimes\Nrp_\gamma$ in $W_q(F)$. Moreover, as explained in~\cite[(16.18)]{Knus:1998}, there exists a unique $3$-fold Pfister form $j_\gamma(\gamma')$ such that $$j_\gamma(\gamma')\equiv \pff{\Nrp_\gamma(x)} \otimes \Nrp_\gamma \mod I^4_q(F).$$ The isomorphism class of this $3$-fold Pfister form is called the relative discriminant of $\gamma'$ with respect to $\gamma$. 
By definition, $j_\gamma(\gamma')$ is hyperbolic if $\gamma$ and $\gamma'$ are conjugate. In fact, the relative discriminant classifies symplectic involutions on $A$, as explained in~\cite[Thm(16.19)]{Knus:1998}. That is, two involutions $\gamma'$ and $\gamma''$ are conjugate if and only if $j_\gamma(\gamma')$ and $j_\gamma(\gamma'')$ are isomorphic. 

With this in hand, we can now prove the main results of this section. 

\begin{proof}[Proof of proposition~\ref{prop:containq}]
Assume that  $A=Q'\otimes_F Q$ and $\sigma$ is a symplectic involution on $A$ which is hyperbolic over $F_Q$. 
Let $\gamma=\ba\otimes \ba$ be the tensor product of the two canonical involutions. It is a symplectic involution on $A$, and as we just recalled, there exists an invertible element $x\in\Symd(A,\gamma)$ such that $\sigma=\Int(x)\circ \gamma$. 
For any $F$-linear involution $\tau$ on $Q'$, there exists $y\in\Sym(Q',\ba)$ such that $\tau=\Int(y)\circ \ba$, so that $\ba\otimes\tau=\Int(1\otimes y)\circ \gamma$. Hence we need to prove that there exists $y\in\Sym(Q',\ba)$ such that 
$\Int(x)\circ\gamma$ and $\Int(1\otimes y)\circ\gamma$ are conjugate. 

In order to prove this, let us consider the $3$-fold Pfister form $j_\gamma(\sigma)$. 
Using the identity~(\ref{qfidentity}), one may easily check from the computation made in~\cite[(16.15)]{Knus:1998}, with $v_1=1$, that the Pfaffian reduced norm of $(A,\gamma)$ is $\Nrp_\gamma=n_Q\perp n_{Q'}\in W_q(F)$. 
Hence, we get $$j_\gamma(\sigma)\equiv \pff{\Nrp_\gamma(x)}\otimes (n_Q\perp n_{Q'})\mod I_q^4(F).$$
Both $\gamma$ and $\sigma$ are hyperbolic over $F_Q$, hence this quadratic form is killed by $F_Q$. 
Since $n_Q$ is hyperbolic over $F_Q$, the Arason-Pfister Hauptsatz~\cite[(23.7)]{Elman:2008} shows that $\pff{\Nrp_\gamma(x)}\otimes n_{Q'}$ also is hyperbolic over $F_Q$. Therefore by Lemma~\ref{prop:pfisterneigh} there exists $\lambda\in F^\times$ such that 
$\pff{\Nrp_\gamma(x)}\otimes n_{Q'}=\pff{\lambda}\otimes n_Q$. Moreover, using the common slot lemma~\ref{prop:commonslot}, we may even assume that 
$$\pff{\Nrp_\gamma(x)}\otimes n_{Q'}=\pff{\lambda} \otimes n_{Q'}=\pff{\lambda}\otimes n_Q.$$
Hence, applying Lemma~\ref{Pfisteridentity}, we get $$j_\gamma(\sigma)\equiv \pff{\Nrp_\gamma(x)}\otimes(n_Q\perp n_{Q'})=\pff{\lambda\Nrp_\gamma(x)}\otimes n_Q\mod I_q^4(F),$$
so, by uniqueness of $j_\gamma(\sigma)$, we have $j_\gamma(\sigma)\simeq \pff{\lambda\Nrp_\gamma(x)}\otimes n_Q$. 
In addition, the previous equalities also show that $\pff{\lambda}\otimes(n_Q\perp n_Q')$, which is Witt equivalent to $\pff{\lambda} \otimes \Nrp_\gamma$, is hyperbolic. Hence, there exists $\tilde x\in \Symd(A,\gamma)$ such that $\lambda\Nrp_\gamma(x)=\Nrp_\gamma(\tilde x)$, and we get $$j_\gamma(\sigma)=\pff{\Nrp_\gamma(\tilde x)} \otimes n_Q.$$
Since $\pff{\Nrp_\gamma(x)}\otimes n_{Q'}=\pff{\lambda} \otimes n_{Q'}$, the form $\pff{\Nrp_\gamma(\tilde x)}\otimes n_{Q'}$ is hyperbolic. 

We claim that the result now follows from the following lemma: 
\begin{lemma} \label{lemma:big}
Let $Q'$ be a quaternion algebra and $(A,\gamma)=(Q,\can)\otimes (Q',\can)$.  Assume $\tilde x\in \Symd(A,\gamma)$ is a non-zero element such that $\pff{\Nrp_\gamma(\tilde x)}\otimes n_{Q'}$ is hyperbolic. Then the quadratic forms  
$\qf{\Nrp_\gamma(\tilde x)}\otimes n_Q$ and $n_{Q'}^0$ do represent a common value, where $n_{Q'}^0$ denotes the restriction of the norm form of $Q'$ to pure quaternions. 
\end{lemma}
Indeed, assuming the lemma for now, we get that there exists a quaternion $y\in Q$ and a pure quaternion $y'\in {Q'}^0$ such that 
$\Nrp_\gamma(\tilde x)n_Q(y)=n_{Q'}(y')$. 
Since $\pff{n_Q(y)} \otimes n_Q$ and $\pff{n_{Q'}(y')}\otimes n_{Q'}$ are hyperbolic, we get, using again Lemma~\ref{Pfisteridentity}, 
\begin{eqnarray*}j_\gamma(\sigma)&\equiv& \pff{\Nrp_\gamma(\tilde x)n_Q(y)} \otimes n_Q
\\&\equiv&\pff{n_{Q'}(y')}\otimes  (n_Q\perp n_{Q'})
\\&\equiv& \pff{n_{Q'}(y')} \otimes \Nrp_\gamma\mod I_q^4(F)\,.\end{eqnarray*}
On the other hand, since $y'$ is pure quaternion, $y'\in \Symd(A,\gamma)$ and $n_{Q'}(y')=\Nrp_\gamma(1\otimes y')$. 
Hence, $j_\gamma(\sigma)=j_\gamma\bigl(\Int(1\otimes y')\circ \gamma\bigr)$, the two involutions are conjugate and this finishes the proof. 
\end{proof} 

Before proving the lemma, we  explain how we deduce the main theorem. 
\begin{proof}[Proof of theorem~\ref{thm:sympclass}]
Let $(A,\sigma)$ be an algebra with symplectic anisotropic involution which is hyperbolic over $F_Q$. If the algebra $A$ contains $Q$, then $(A,\sigma)$ is as in (a) by Proposition~\ref{prop:containq}. Otherwise, $A=M_2(Q')$ for some quaternion algebra $Q'$ over $F$ such that $Q\otimes_F Q'$ is division. Arguing as in the proof of Proposition~\ref{prop:splithyp}(2), one may check there exists $\mu\in F^\times$ such that $(A,\sigma)=\Ad_{\pff{\mu}}\otimes(Q', \ba)$. Again, let us denote by $\gamma$ the involution $\ba\otimes \ba$, which is symplectic and hyperbolic. We have $\Nrp_\gamma=n_{Q'}$ and $$j_\gamma(\sigma)=\pff{\mu} \otimes n_{Q'}.$$
Since $\sigma$ is hyperbolic over $F_Q$, this form is split by $F_Q$, and applying again Lemma~\ref{prop:pfisterneigh} and the common slot lemma~\ref{prop:commonslot}, we get 
$$j_\gamma(\sigma)=\pff{\lambda}\otimes n_Q=\pff{\lambda}\otimes n_{Q'},$$ for some $\lambda\in F^\times$. 
Moreover, since the relative discriminant is classifying, as explained above, we have 
$(A,\sigma)\simeq \Ad_{\pff{\lambda}}\otimes(Q', \ba)$ and the proof is complete. 
\end{proof}

We finish this section with the core of the argument, which is hidden in lemma~\ref{lemma:big}. 

\begin{proof}[Proof of Lemma~\ref{lemma:big}]
Let us denote $\mu=\Nrp_\gamma(\tilde x)$. We have to prove that the quadratic forms $\qf{\mu}\otimes n_Q$ and $n_{Q'}^0$ represent a common value. 
If either $Q$ or $Q'$ is not division, then one of the two quadratic forms is isotropic, hence universal, and the result follows immediately. Hence we may assume that $Q$ and $Q'$ are both division. 

Let $(1,u,v,w)$ be a quaternion basis of $Q$ and $(1,u',v',w')$ a quaternion basis of $Q'$.
Since $\tilde x\in\Symd(A,\gamma)$, the computation in~\cite[(16.15)]{Knus:1998} shows there exist $\alpha,\beta,\zeta,\delta,\lambda,\nu\in F$ such that 
$$\tilde x = \alpha(1\otimes 1) +\beta( u\otimes 1 + 1\otimes u') +\zeta (v\otimes 1) + \delta (1\otimes v') +\lambda (w\otimes 1) +\nu (1\otimes w').$$
We claim that, after a change of quaternion basis for $Q$ and $Q'$ as in~Lemma~\ref{prop:quatbasischange}, we may assume that either $\beta\not=0$ and $\zeta=\delta=\lambda=\nu=0$, or 
$\beta=\delta=\nu=0$ and $\zeta, \lambda\in\{0,1\}$.  
Indeed, assume first $\beta\not=0$. Then we may replace $u$ by $\hat{u} = u + \frac{\zeta}{\beta} v + \frac{\lambda}\beta w$ and $u'$ by $\hat{u'} = u' + \frac{\delta}{\beta} v' + \frac{\nu}\beta w'$, so that $\tilde x= \alpha(1\otimes 1) +\beta( \hat u\otimes 1 + 1\otimes\hat {u'})$. 
Assume now $\beta=0$. 
If $(\zeta v+\lambda w)^2=0$, those two terms do not contribute to $\mu=\Nrp_\gamma(\tilde x)$, and we may assume $\zeta=\lambda=0$. Otherwise, we change the quaternion basis of $Q$, replacing $v$ with $\hat v=\zeta v+\lambda w$, so that $\tilde x= \alpha(1\otimes 1) +\hat v\otimes 1 + \delta (1\otimes v') +\nu (1\otimes w')$. The same argument applied to $Q'$ shows we may also assume $\nu=0$ and $\delta\in\{0,1\}$. 

These changes of basis being made, we let $u^2+u=a$, $v^2=b$, $u'^2+u'=a'$ and $v'^2=b'$, so that $n_Q=\pfr{a,b}$ and $n_{Q'}^0=\qf{1}\perp b'[1,a']$. By our hypothesis, the form $\pfr{\mu,a',b'}$ is hyperbolic. Adding $2\times[1,a+a']\simeq 2\times \HH$, we get that 
$$\pfr{\mu,b',a'}\perp 2\times [1,a+a'] \simeq 6\times \HH. $$
Since $[1,a']\perp[1,a+a']=[1,a]\perp\HH$ by~(\ref{eqnarray:nonsingiso}), Witt cancellation~\cite[(8.4)]{Elman:2008} gives 
\begin{eqnarray}
\label{hyp}
  [1,a]\perp b'[1,a']\perp \mu[1,a']\perp \mu b'[1,a']\perp [1,a+a']\simeq 5\times \HH\,.  
\end{eqnarray}
The scalar $\mu=\Nrp_\gamma(\tilde x)$ is given either by $\mu=\alpha^2+\alpha\beta+\beta^2(a+a')$, with $\beta\not=0$, or by $\mu=\alpha^2+\varepsilon b+\varepsilon' b'$, with $\varepsilon,\varepsilon'\in \{0,1\}$. 
Let us now consider the different possible cases separately. 

\textbf{Case $(1)$:} Assume $\tilde x = \alpha(1\otimes 1) +\beta( u\otimes 1 + 1\otimes u')$, and $\mu=\alpha^2 +\alpha\beta +(a+a')\beta^2$, with $\beta\not=0$. 
The quadratic form $[1,a+a']\perp \qf{\mu}$ is isotropic and hence isometric to $\HH\perp \qf{\mu}$ by (\ref{eqnarray:isoiso}). 
Therefore, by Proposition~\ref{prop:nsc}, there exists an element $c\in F$ such that $[1,a+a']\perp {\mu}[1,a']\simeq \HH \perp \mu[1,c]$. Substituting this into (\ref{hyp}), and using   Witt cancellation gives 
$$ \rho= [1,a]\perp b'[1,a']\perp \mu[1,c]\perp \mu b'[1,a']\simeq 4\times \HH\,.  $$
The hyperbolic quadratic  form $\rho$ is  $8$-dimensional, hyperbolic and dominates the $5$-dimensional quadratic  form  $\psi=[1,a]\perp \mu b'[1,a']\perp\qf{\mu}$. Hence, $\psi$ is isotropic and therefore $[1,a]$ and $ \mu b'[1,a']\perp\qf{\mu}$ represent an element in common. As $[1,a]$ is a subform of $n_Q$ and $ \mu b'[1,a']\perp\qf{\mu}\simeq \qf{\mu}\otimes n_{Q'}^0$, this gives the result.

\textbf{Case $(2)$:}  Assume $\tilde x=\alpha(1\otimes 1)+\varepsilon v+\varepsilon' v'$ with $(\varepsilon,\varepsilon')\not=(1,1)$. This gives three subcases in which  we easily get the required conclusion. 
If $\tilde x=\alpha(1\otimes 1)$ then $\mu$ is a square, and the result follows since both $n_Q$ and $n_{Q'}^0$ represent $1$. 
If $\tilde x=\alpha(1\otimes 1)+v$, so that $\mu=\alpha^2+ b$, then $n_Q=\pfr{a,b}$ represents $\mu$. Hence $\qf{\mu}\otimes n_Q$ represents $1$, and the result follows.
Finally if $\tilde x=\alpha(1\otimes 1)+v'$ so that $\mu = \alpha^2 + b'$, then $\mu$ is represented by $n_Q'^0$, and also by $\qf{\mu}\otimes n_Q$. 

\textbf{Case $(3)$:} The only remaining case is $\tilde x=\alpha(1\otimes 1)+v+v'$ and $\mu = \alpha^2 +b+b'$. Using~(\ref{eqnarray:singiso})  and~(\ref{eqnarray:multiply}), one may check that 
$$ b'[1,a']\perp\qf{\mu} \simeq  (b'+\mu)[1,c] \perp\qf{\mu} \textrm{ for some } c\in F\,.$$
Since $b'+\mu=b+\alpha^2$, adding $\qf{\mu}$ on both sides of~(\ref{hyp}) we get 
$$   [1,a]\perp (b+\alpha^2)[1,c]\perp \mu[1,a']\perp \mu b'[1,a']\perp [1,a+a']\perp\qf{\mu}\simeq 5\times \HH\perp\qf{\mu}\,.  $$
We consider two cases.

\textbf{Case $(a)$:} $\alpha=0$. We have by (\ref{eqnarray:singiso}) that
$ \mu[1,a']\perp \qf{\mu}\simeq \HH \perp \qf{\mu}$
 and hence substituting this into the above equation we get 
$$  \rho= [1,a]\perp b[1,c]\perp  \mu b'[1,a']\perp [1,a+a']\perp\qf{\mu}\simeq 4\times \HH\perp\qf{\mu}\,.  $$ 
The $9$-dimensional form $\rho$ has a totally isotropic subspace of dimension $4$ and dominates the $6$-dimensional quadratic form $\psi=[1,a]\perp \qf{b}\perp \mu b'[1,a']\perp\qf{\mu}$. Therefore as in Case $(1)$, $\psi$ is isotropic. Hence $[1,a]\perp \qf{b}$, which is dominated by $n_Q$, and  $\qf{\mu}\otimes n_{Q'}$ do represent a common value, as required.

\textbf{Case $(b)$:} $\alpha\neq 0$. Using (\ref{eqnarray:multiply}) and (\ref{eqnarray:singiso}), we get that for some $d,e,f\in F$
\begin{eqnarray*}
(b+\alpha^2)[1,c] \perp  [1,a+a']  & \simeq  &  (b+\alpha^2)  [1,c] \perp [\alpha^2 , (a+a')\alpha^{-2}] \\
& \simeq & b[1,d] \perp [\alpha^2 ,e] \simeq  b[1,d]\perp [1,f]\,.
\end{eqnarray*}
Substituting this into the equation above gives
$$   [1,a]\perp b[1,d]\perp \mu[1,a']\perp \mu b'[1,a']\perp [1,f]\perp\qf{\mu}\simeq 5\times \HH\perp\qf{\mu}\,.  $$
Again by  (\ref{eqnarray:singiso})  we have that 
$ \mu[1,a']\perp\qf{ \mu}\simeq \HH \perp \qf{\mu}$ and hence we get 
$$ \rho=  [1,a]\perp b[1,d]\perp  \mu b'[1,a']\perp [1,f]\perp\qf{\mu}\simeq 4\times \HH\perp\qf{\mu}\,.  $$
The $9$-dimensional form $\rho$ has a totally isotropic subspace of dimension $4$ and dominates the $6$-dimensional quadratic for $\psi=[1,a]\perp \qf{b}\perp \mu b'[1,a']\perp\qf{\mu}$, and hence we get the result as in Case $(a)$.
\end{proof}

\subsection{$F_Q$-minimal forms of dimension $5$}\label{sec:minform}

In this section we use Theorem~\ref{thm:sympclass} to recover a result of Faivre, which characterises $F_Q$-minimal quadratic forms of dimension $5$. The proof is based on the exceptional isomorphism $B_2\equiv C_2$, described in~\cite[\S 15.C]{Knus:1998}. More precisely, to a $5$-dimensional nondegenerate quadratic form $\rho$, one may associate its even Clifford algebra $C_0(\rho)$, which is a biquaternion algebra~\cite[\S 11]{Elman:2008}. Moreover, the canonical involution $\tau_0$ restricted to $C_0(\rho)$ is of symplectic type. Conversely, any biquaternion algebra with symplectic involution $(A,\sigma)$ is the Clifford algebra of a $5$ dimensional nondegenerate quadratic form, endowed with its canonical involution. This quadratic form can be explicitly described: it is the restriction of the reduced Pfaffian norm $\Nrp_\sigma$ to the set $\Symd(A,\sigma)^0$ of symmetrised elements with Pfaffian reduced trace $0$. 

In the previous section we described algebras of degree $4$ with symplectic involution that are hyperbolic over $F_Q$. 
To deduce information on $F_Q$-minimal quadratic forms we will use the following: 
\begin{prop}\label{prop:formtocifford}
Let $\rho$ be a nondegenerate quadratic form over $F$ of odd dimension. 
\begin{enumerate}[$(1)$]
\item If $\rho$ is isotropic then $(C_0(\rho),\tau_0)$ is hyperbolic. If $\dim(\rho)=5$ the converse holds. 
\item If $\rho$ dominates a quadratic form similar to $[1,a]\perp\qf{b}$ then $(C_0(\rho),\tau_0)$ contains $(Q,\can)$. The converse holds if $\dim(\rho)=5$. 
\end{enumerate}
\end{prop}
\begin{proof}
$(1)$ The first statement follows from \cite[(8.5)]{Knus:1991} and the second from \cite[(15.21)]{Knus:1991}.

$(2)$  Assume $\rho$ dominates the quadratic form $[1,a]\perp\qf{b}$. Then  the underlying vector space of $\rho$ contains vectors mapping to elements  $e_1,e_2$ and $e_3$ in  the full Clifford algebra of $\rho$, such that 
$$ e_1^2= 1,\quad  e_2^2=a, \quad e_3^2=b, $$
$$ e_1e_2+e_2e_1=1 \quad \textrm{and} \quad e_ie_3=e_3e_i\quad\textrm{for all } i\in \{1,2\}\,. $$
Hence, $(e_1e_2)^2=e_1e_2+a$, $(e_1e_3)^2=b$ and $(e_1e_2)(e_1e_3)=(e_1e_3)(1+e_1e_2)$. 
Moreover, the action of the canonical involution is given by $\tau_0(e_1e_2)=e_2e_1= 1+e_1e_2$ and $\tau_0(e_1e_3)=e_3e_1=e_1e_3$. 
Therefore $(e_1e_2, e_1e_3)$ generate a stable subalgebra of the even Clifford algebra $(C_0(\rho),\tau_0)$ isomorphic to $(Q,\ba)$. 
Since $(C_0(\rho),\tau_0)$ only depends on the similarly class of $\rho$, we have the first implication.

Conversely, assume $\rho$ has dimension $5$ and $(C_0(\rho),\tau_0)$ contains $(Q,\can)$. Since $C_0(\rho)$ has degree $4$, it decomposes as 
$$(C_0(\rho),\tau_0)\simeq (Q',\tau)\otimes (Q,\can)$$ for some quaternion algebra with involution $(Q',\tau)$.  We need to prove that the restriction of $\Nrp_{\tau_0}$ to the subset $\Symd(C_0(\rho),\tau_0)^0$, which consists of symmetrised elements with Pfaffian reduced trace $0$, contains a subform similar to $[1,a]\perp\qf{b}$. The quadratic form $\Nrp_{\tau_0}$ is computed in~\cite[(15.19)]{Knus:1998}. Let $y$ be a pure quaternion in $Q'$ such that $\tau=\Int(y)\circ \can$, and consider a quaternion basis $(1,u,v,w)$ of $Q$; the elements $1\otimes y$, $v\otimes y$ and $w\otimes y$ are symmetrised  elements with trivial Pfaffian reduced trace, and they generate a subspace of $\Symd(C_0(\rho),\tau_0)^0$, over which $\Nrp_{\tau_0}$ is $y^2(\qf{1}\perp b[1,a])$, which is similar to $\qf{b}\perp [1,a]$ as required. 
\end{proof}
The following was first shown in~\cite[(5.2.12)]{fairve:thesis}, and extends~\cite[Prop.~4.1]{hoffmann:minforms} to fields of characteristic $2$. 
\begin{cor} A $5$-dimensional nondegenerate quadratic form $\rho$ over $F$ is $F_Q$-minimal if and only if the following conditions hold:
\begin{enumerate}[$(a)$]
\item $\rho$ is similar to a Pfister neighbour of $\pfr{\lambda,b,a}$ for some $\lambda\in F^\times$.
\item$C_0(\rho)\simeq M_2(Q')$ for some $F$-quaternion algebra $Q'$ such that $Q\otimes_F Q'$ is a division algebra. 
\end{enumerate}
\end{cor}
\begin{proof} 
 An anisotropic quadratic form of dimension $4$ or less becomes isotropic over $F_Q$ if and only if it dominates a form similar to   $\qf{b}\perp[1,a]$, see  \cite[\S 1.3 and \S 1.4]{laghribi:qfdim6}.
Hence the  quadratic form $\rho$ is $F_Q$-minimal if and only if it is anisotropic, isotropic over $F_Q$ and does not dominate any form similar to $\qf{b}\perp[1,a]$. 
By Proposition~\ref{prop:formtocifford} this holds if and only if $(C_0(\rho),\tau_0)$ is non-hyperbolic, hyperbolic over $F_Q$ and does not contain $(Q,\can)$.  Since we are in the degree $4$ symplectic case, non-hyperbolic and anisotropic are equivalent. Therefore  applying Theorem~\ref{thm:sympclass}, 
 we get that $\rho$ is $F_Q$-minimal if and only if 
$$ (C_0(\rho),\tau_0) \simeq \Ad_{\pff{\lambda}}\otimes (Q',\can)$$
for some $F$-quaternion algebra $Q'=[a',b')_F$ and some $\lambda\in F^\times$ with $Q\otimes_F Q'$ a division algebra and $\pfr{\lambda,b,a}\simeq  \pfr{\lambda,b',a'}$. 

It follows from the above isomorphism that $\rho$ is similar to a Pfister neighbour of $\pfr{\lambda,b',a'}$ by \cite[Exercise 8, p.270]{Knus:1998}. Thus $(a)$ and $(b)$ hold if $\rho$ is $F_Q$-minimal.
Conversely, if $\rho$ satisfies condition (a) it is isotropic over $F_Q$, hence $(C_0(\rho),\tau_0)$ is hyperbolic over $F_Q$. Moreover, condition (b) guarantees that $A$ does not contain $Q$. Therefore $ (C_0(\rho),\tau_0) $ does not contain $(Q,\can)$ by Theorem~\ref{thm:sympclass}, and $\rho$ is $F_Q$-minimal. 
\end{proof}

\section{Quadratic pairs over the function fields of conics}\label{sec:qp}

Throughout this section $(A,\sigma,f)$ is a central simple algebra endowed with a quadratic pair. 
We will describe the quadratic pairs that become hyperbolic over $F_Q$ if either $A_{F_Q}$ is split or $A$ has degree $4$.  
Quadratic pairs over fields of characteristic different from $2$ are equivalent to orthogonal involutions, and hence 
the results and proofs here are very similar to those of~\cite{queguiner:conic} on orthogonal involutions. The  additional ingredient we use in characteristic $2$ is the existence of a canonical quadratic pair on a tensor product of two algebras with symplectic involutions. 

\subsection{Quadratic pairs that become split hyperbolic over $F_Q$}

A central simple algebra split by $F_Q$ is either split or Brauer equivalent to $Q$. 
In both cases we have a complete description of the quadratic pairs that are hyperbolic over $F_Q$:

\begin{prop}
Let $(A,\s,f)$ be an $F$-algebra with quadratic pair. \begin{enumerate}[$(1)$]
 \item Assume $A$ is split, so that $(A,\sigma,f)\simeq \Ad_\rho$ for some nonsingular quadratic form $\rho$ over $F$. 
 Then $(A,\sigma,f)$ is hyperbolic over $F_Q$ if and only if the anisotropic part of $\rho$ is a multiple of $n_Q$, that is $\rho_\an\simeq \varphi\otimes n_Q$ for some symmetric bilinear form $\varphi$ over $F$. 
Moreover, assuming  $({\Ad_\rho})_{F_Q}$ is hyperbolic,  then $\Ad_\rho$ contains $(Q,\can)$ if and only if the Witt index of $\rho$ is a multiple of $4$.
 
\item  If $A$ is Brauer-equivalent to $Q$ and $(A,\s,f)_{F_Q}$ is hyperbolic then $(A,\s,f)$ is hyperbolic and contains $(Q,\can)$.
 \end{enumerate}
\end{prop}
\begin{remark}
In particular,  an anisotropic algebra with involution which is split and hyperbolic over $F_Q$ is already  split over $F$, and does contain $(Q,\can)$. 
\end{remark}
\begin{proof} 
(1) Assume $A$ is split and let $\rho$ be a quadratic form over $F$ such that $(A,\sigma, f)=\Ad_\rho$. The first assertion follows from  a particular case of the subform theorem, recalled in Lemma~\ref{prop:pfisterneigh}. Therefore to prove the second assertion we may assume $\rho=\varphi\otimes n_Q\perp r\HH$ for some symmetric bilinear form $\varphi$ and  $r$ the Witt index of $\rho$. 
If $r=4s$ for some integer $s$ then $r\HH\simeq (s\HH^{bi})\otimes n_Q$, therefore, by Lemma~\ref{decAdnQ.lem}, 
$$\Ad_{\rho}\simeq \Ad_{(\varphi\perp s\HH^{bi})} \otimes (Q,\can) \boxtimes (Q,\can)\,.$$ 
Hence $(A,\sigma,f)$ contains $(Q,\can)$ in this case. 
Assume conversely that $(A,\sigma,f)$ contains $(Q,\can)$, that is $(A,\sigma,f)\simeq (Q,\can)\boxtimes (B,\tau)$ for some algebra $B$ Brauer-equivalent to $Q$.
By Proposition~\ref{prop:splithyp}, $(B,\tau)\simeq\Ad_b\otimes(Q,\can)$ for some symmetric bilinear form $b$, hence $(A,\sigma,f)\simeq \Ad_{b}\otimes (Q,\can)\boxtimes (Q,\can)\simeq \Ad_{b\otimes n_Q}$
by  Lemma~\ref{decAdnQ.lem}.

(2) Assume now $A$ is Brauer equivalent to $Q$. 
By~\cite[(8.2)]{dolphin:quadpairs}, the quadratic pair $(\sigma,f)$ is hyperbolic over $F_Q$ if and only if it is hyperbolic over $F$. 
If this is the case then $A$ has even coindex, $A\simeq M_{2r}(Q)$ for some nonnegative integer $r$. 
So the tensor product $(Q,\can)\otimes \Ad_{(r\HH)}$ is Brauer equivalent to $A$, of the same degree, and also hyperbolic. Hence it is isomorphic to $(A,\sigma,f)$, and $(A,\sigma,f)$  contains $(Q,\can)$ as required. 
\end{proof}

\subsection{Quadratic pairs on algebras of degree $4$}
Let $(A,\sigma,f)$ be an $F$-algebra with quadratic pair. If $A$ has even degree, one may associate to $(\sigma,f)$ a discriminant and a Clifford algebra, which is endowed with a canonical involution~\cite[\S\,7.B]{Knus:1998}. The discriminant has values in $F/\wp(F)$, where 
$\wp(F)=\{x^2+x\mid x\in F\}$. It defines a quadratic extension $K/F$, which is the center of the Clifford algebra~\cite[Thm (8.10)]{Knus:1998}. In particular, if $(\sigma,f)$ has trivial discriminant, then its Clifford algebra is a direct product of two central simple algebras $C_+\times C_-$. 

Assume now that $A$ has degree $4$. Then the Clifford algebra, as an algebra with involution, is a complete invariant of the algebra with quadratic pair. This follows from the exceptional isomorphism $A_1^2\equiv D_2$, and is explained in detail in~\cite[\S\,15.B]{Knus:1998}. In particular, if $(A,\sigma,f)$ has trivial discriminant then its Clifford algebra is a direct product of two quaternion algebras $H_+$ and $H_-$, each endowed with its canonical involution and 
$$(A,\sigma)\simeq (H_+,\can)\boxtimes (H_-,\can)\mbox{ see~\cite[(15.12)]{Knus:1998}}.$$
With this in hand we now prove: 
\begin{prop}
Let $(A,\s,f)$ be an $F$-algebra of degree $4$ with a non-hyperbolic quadratic pair. Then $(A,\s,f)_{F_Q}$ is hyperbolic if and only if $(A,\s,f)$ contains $(Q,\can)$.
\end{prop}
\begin{proof}
The reverse implication is a special case of \cite[A.5]{Tignol:galcohomgps}. So let us assume $(A,\sigma,f)$ is hyperbolic over $F_Q$. 
By~\cite[\S 8.E]{Knus:1998}, $(\sigma,f)_{F_Q}$ has trivial discriminant and one component of its Clifford algebra is split. Since $F/\wp(F)$ embeds in $F_Q/\wp(F_Q)$, the discriminant of $(\sigma,f)$ is trivial already over the base field $F$. Hence $(A,\sigma,f)\simeq (H_+,\can)\boxtimes (H_-,\can)$. Moreover, one of the two components, say $H_+$, is split over $F_Q$, hence $H_+$ is either split or isomorphic to $Q$. Since $(M_2(F),\can)\boxtimes( H_-,\can)\simeq (M_2(F),\can)\boxtimes (H_-,\can)$ is hyperbolic, $H_+$ cannot be split. Therefore $H_+=Q$ and  we have 
$(A,\sigma,f)\simeq (Q,\can)\boxtimes (H_-,\can).$\end{proof}

\small{}

\end{document}